\documentclass[12pt]{amsart}

\font\bbbld=msbm10 scaled\magstephalf

\newcommand{\bi}{\bar{i}}
\newcommand{\bj}{\bar{j}}
\newcommand{\bk}{\bar{k}}
\newcommand{\bl}{\bar{l}}
\newcommand{\bm}{\bar{m}}
\newcommand{\bn}{\bar{n}}

\newcommand{\bq}{\bar{q}}

\newcommand{\bz}{\bar{z}}

\newcommand{\bM}{\bar{M}}

\newcommand{\balpha}{\bar{\alpha}}

\newcommand{\bzeta}{\bar{\zeta}}

\newcommand{\bpartial}{\bar{\partial}}

\newcommand{\fg}{\mathfrak{g}}
\newcommand{\fI}{\mathfrak{I}}

\newcommand{\fRe}{\mathfrak{Re}}

\newcommand{\bfG}{\hbox{\bbbld G}}

\newcommand{\bfN}{\hbox{\bbbld N}}

\newcommand{\bfR}{\hbox{\bbbld R}}

\newcommand{\cC}{\mathcal{C}}

\newcommand{\cP}{\mathcal{P}}

\newcommand{\tr}{\mbox{tr}}

\newcommand{\va}{{\bf a}}

\newcommand{\ol}{\overline}
\newcommand{\ul}{\underline}

\newtheorem{theorem}{Theorem}[section]
\newtheorem{lemma}[theorem]{Lemma}

 \theoremstyle{definition}
\newtheorem{definition}[theorem]{Definition}

\theoremstyle{remark}
\newtheorem{remark}[theorem]{Remark}

\numberwithin{equation}{section}

\begin{document}

\title[fully nonlinear elliptic equations]
{Fully nonlinear elliptic equations on Hermitian manifolds for symmetric functions \\
of partial Laplacians}
\author{Mathew George}
\address{Department of Mathematics, Ohio State University,
         Columbus, OH 43210, USA}
\email{george.924@buckeyemail.osu.edu}
\author{Bo Guan}
\address{Department of Mathematics, Ohio State University,
         Columbus, OH 43210, USA}
\email{guan@math.ohio-state.edu}
\author{Chunhui Qiu}
\address{School of Mathematical Sciences, Xiamen University,  Xiamen, 361005, China}
\email{chqiu@xmu.edu.cn}
\thanks{The third author was partially supported by NSFC (Grant No. 11971401).}
\date{}

\begin{abstract}
We consider a class of fully nonlinear second order elliptic equations on Hermitian manifolds closely related to the general notion of $\bfG$-plurisubharmonicity of Harvey-Lawson 
and an equation treated by Sz\'ekelyhidi-Tosatti-Weinkove in the proof of Gauduchon conjecture.
Under fairly general assumptions we derive interior estimates and establish the existence of smooth solutions for the Dirichlet problem as well as for equations on closed manifolds.

{\em Mathematical Subject Classification (MSC2020):}
35J15, 35J60, 58J05.
\end{abstract}

\maketitle

\section{Introduction}

\medskip

In a remarkable paper where the Gauduchon conjecture which extends Calabi-Yau Theorem 
to non-K\"ahler metrics was proved,
Sz\'ekelyhidi-Tosatti-Weinkove~\cite{STW17} solved the Monge-Amp\`ere type equation
\begin{equation}
\label{ggq-I10}
  (\Delta u \omega - \sqrt{-1} \partial \bpartial u +  \chi (\partial u, \bpartial u))^n 
      = \psi \omega^n  
\end{equation}
on a compact Hermitian manifold $(M^n, \omega)$, where for a function $u \in C^{2} (M)$, $\chi (\partial u, \bpartial u)$ is a real $(1,1)$ form which depends on
$\partial u$, $\bpartial u$ linearly and satisfies some additional assumptions
(which for convenience we shall refer as {\em STW conditions}). 
This equation also has close connections to the form-type Monge-Amp\`ere equations studied by
Fu-Wang-Wu~\cite{FWW11, FWW15}, and the notion of $n-1$ plurisubharmonic functions
in the sense of  Harvey-Lawson~\cite{HL11, HL12, HL13}.

In the same paper \cite{STW17}, the authors also treated more general equations of the form 
\begin{equation}
\label{ggq-I20}
  f (\lambda( \Delta u \omega - \sqrt{-1} \partial \bpartial u +  \chi (\partial u, \bpartial u))
      = \psi   
\end{equation}
under the same (STW) conditions on $\chi$, where $\lambda(U)$ denotes the eigenvalues of a real $(1, 1)$ form $U$ with respect to $\omega$, and 
$f$ is a symmetric function of $n$ variables satisfying the structure conditions of 
 Caffarelli-Nirenberg-Spruck~\cite{CNS3}, i.e.
$f$ is assumed to be defined in a symmetric
open and convex cone $\Gamma \subset \bfR^n$ with vertex
at the origin, 
\begin{equation}
\label{3I-15}
\Gamma_n = \{\lambda \in \bfR^n: \lambda_i > 0\} \subset\Gamma,
\end{equation}
and satisfies the conditions
\begin{equation}
\label{3I-20}
f_i 
 \equiv \frac{\partial f}{\partial \lambda_i} > 0 \;\;
\mbox{in $\Gamma$}, \;\; 1 \leq i \leq n,
\end{equation}
\begin{equation}
\label{3I-30}
\mbox{$f$ is a concave function in $\Gamma$},
\end{equation}
\begin{equation}
\label{3I-40}
 \sup_{\partial \Gamma} f < \inf_M \psi 
\end{equation}
and, in addition
\begin{equation}
\label{3I-45}
\lim_{t \rightarrow \infty} f (t \lambda) = \sup_{\Gamma} f, \;\; \forall \; \lambda \in \Gamma. 
\end{equation}

There are other important fully nonlinear equations of second order in complex geometry with explicit dependence on the solution and its gradient,
such as the Fu-Yau equation~\cite{FY07, FY08} arising from the Strominger system in superstring theory and its higher dimension extensions studied by Phong-Picard-Zhang~\cite{PPZ16, PPZ17, PPZ21} and Chu-Huang-Zhu~\cite{CHZ19}. 
In general, the presence of these terms and the gradients especially, cause substantial difficulties in solving the equations, even in the case of linear dependence for the Monge-Amp\`ere type equation~\eqref{ggq-I10}.


It is clear that Sz\'ekelyhidi-Tosatti-Weinkove \cite{STW17}  primarily focused on 
equation~\eqref{ggq-I10} in order to prove the Gauduchon conjecture. It seems worthwhile to take a closer look at equation~\eqref{ggq-I20} from different angles, and from a more PDE point of view. This is one of our main purposes 
in the current paper. 
We shall assume 
in place of \eqref{3I-20} that $f$ satisfies 
the weaker condition
\begin{equation}
\label{3I-20w}
f_i \equiv 
    \frac{\partial f}{\partial \lambda_i} \geq 0 \;\;
\mbox{in $\Gamma$}, \;\; 1 \leq i \leq n.
\end{equation}

In order to describe our results, we recall some notions from \cite{Guan14} and 
\cite{GN}.
For a fixed real number $\sigma \in (\sup_{\partial \Gamma} f, \sup_{\Gamma} f)$
define
\[ \Gamma^{\sigma} = \{\lambda \in \Gamma: f (\lambda) > \sigma\}. \]

\begin{lemma}
\label{ggq-lemma-I0}
 Under conditions \eqref{3I-20w} and \eqref{3I-30}, the level hypersurface of $f$
 \[ \partial \Gamma^{\sigma} = \{\lambda \in \Gamma: f (\lambda) = \sigma\}, \]
which is the boundary of $\Gamma^{\sigma}$, is smooth and convex. 
\end{lemma}

\begin{proof}
If $f$ satisfies both \eqref{3I-20} and \eqref{3I-30}, then
$\partial \Gamma^{\sigma}$ is clearly smooth and convex.
This is still true under the weaker condition~\eqref{3I-20w} in place of \eqref{3I-20}. 
For if $Df (\lambda_0) = 0$ for some $\lambda_0 \in \Gamma$ then 
by \eqref{3I-20w} and \eqref{3I-30} we see that $Df = 0$ in the cone $\lambda_0 + \Gamma_n$ and therefore $f (\lambda_0) = \sup_{\Gamma} f > \sigma$,
showing that $D f \neq 0$ on $\partial \Gamma^{\sigma}$. So $\partial \Gamma^{\sigma}$ is smooth. 
\end{proof}

For $\lambda \in \partial \Gamma^{\sigma}$ let
\[ \nu_{\lambda} = \frac{Df (\lambda)}{|Df (\lambda)|} \]
denote the unit normal vector to $\partial \Gamma^{\sigma}$ at $\lambda$.

\begin{definition}[\cite{Guan14}] 
For $\mu \in \bfR^n$
let 
 \[ S^{\sigma}_{\mu} = \{\lambda \in \partial
\Gamma^{\sigma}: \nu_{\lambda} \cdot (\mu - \lambda) \leq 0\}. \]
The  {\em tangent cone at infinity} 
to $\Gamma^{\sigma}$ is defined as
\[ \begin{aligned}
\cC^+_{\sigma}
 \,& = \{\mu \in \bfR^n:
              S^{\sigma}_{\mu} \; \mbox{is compact}\}.
    \end{aligned} \]
\end{definition}

Clearly $\cC^+_{\sigma}$ is a symmetric convex cone. As in \cite{Guan14} one can show that $\cC^+_{\sigma}$ is open.

\begin{definition}[\cite{GN}]
The {\em rank of ${\mathcal{C}}_{\sigma}^+$} is defined to be
\[ \min \{r (\nu): \mbox{$\nu$ is the unit normal vector of a supporting plane
to ${\mathcal{C}}_{\sigma}^+$}\} \]
where $r (\nu)$ denotes the number of non-zero
components of $\nu$.
\end{definition}


Our first result concerns {\em a priori} interior estimates for second order derivatives.
Let 
$u \in C^4 (M)$ be an admissible solution of the equation %
\begin{equation}
\label{ggq-I20g}
  f (\lambda( \Delta u \omega - \sqrt{-1} \partial \bpartial u +  \chi [u]))
      = \psi [u]  
\end{equation}
where $\chi [u] = \chi (z, u, \partial u, \bpartial u)$ 
and $\psi [u] = \psi (z, u, \partial u, \bpartial u)$ are $C^2$ in  
$(z, u, \partial u, \bpartial u)$.
We assume 
\begin{equation}
\label{ggq-I45}
  \lim_{t \rightarrow +\infty} f (t {\bf{1}}) - \sup_{M} \psi [u] \geq c_0 > 0
\end{equation}
where ${\bf{1}} = (1, \ldots, 1) \in \Gamma$. So $c_0$ may depend on 
$|u|_{C^1 (M)}$. 
This is clearly a necessary and fairly mild condition for the following result to hold.  

\begin{theorem}
\label{ggq-th-I10}
In addition to \eqref{3I-30}, \eqref {3I-40}, \eqref{3I-20w} and \eqref{ggq-I45}
assume that $f$ satisfies
\begin{equation}
\label{ggq-I50}
\sum f_i \lambda_i \geq - C_0  \sum f_i  \;\;
\mbox{in $\Gamma$} 
\end{equation}
for some constant $C_0 \geq 0$, and 
\begin{equation}
\label{ggq-I60}
\mbox{rank of} \;\; \mathcal{C}_{\sigma}^+ \geq 2, 
\;\; \forall \; \inf_{M} \psi [u] \leq \sigma \leq  \sup_{M} \psi [u]. 
\end{equation}
 Let $B_r$ be a geodesic ball of radius $r$ in $M$. Then 
\begin{equation}
\label{ggn-I70}
\sup_{B_{\frac{r}{2}}} |\partial \bpartial u | \leq \frac{C}{r^2} 
\end{equation}
where  $C$ depends  on $c_0$, $C_0$,
$|u|_{C^1 (\bar{B}_r)}$ and other known data.
\end{theorem}

Geometrically, condition \eqref{ggq-I50} means that at any point $\lambda_0 \in \Gamma$, the distance from the origin to the tangent plane at $\lambda_0$
of the level hypersuface $\partial \Gamma^{f (\lambda_0)}$ has a uniform bound $C_0$. It is satisfied in most applications and is weaker than assumption~\eqref{3I-45} which implies
\begin{equation}
\label{ggq-I55}
\sum f_i \lambda_i \geq 0 \;\;
\mbox{in $\Gamma$}. 
\end{equation}

By an inequality of Lin-Trudinger~\cite{LT1994}, 
the rank of $\mathcal{C}_{\sigma}^+$ is $n-k+1$ for $f = \sigma_k^{\frac{1}{k}}$
and $\sigma > 0$, where
\[ \sigma_k (\lambda) = \sum_{1 \leq i_1 < \cdots < i_k \leq n}
     \lambda_{i_1} \cdots \lambda_{i_k} \] 
is the  $k$-th elementary symmetric function defined on the Garding cone
\[ \Gamma_k = \{\lambda \in \bfR^n: \sigma_j (\lambda) > 0, \;
\mbox{for $1 \leq j \leq k$}\}. \]
So Theorem~\ref{ggq-th-I10} applies to $f = \sigma_k^{\frac{1}{k}}$ for $k \leq n-1$ 
but excludes equation~\eqref{ggq-I10} which corresponds to 
$f =  \sigma_n^{\frac{1}{n}}$.

Theorem~\ref{ggq-th-I10} heavily relies on a crucial lemma in Section~\ref{ggq-P}.
From the proof we shall see 
that there exists $\beta_0 > 1$ such that Theorem~\ref{ggq-th-I10} extends to the equation
\begin{equation}
\label{ggq-I20beta}
  f (\lambda( \Delta u \omega - \beta \sqrt{-1} \partial \bpartial u +  \chi [u]))
      = \psi [u],  
\end{equation}
for all $\beta < \beta_0$. 
It would be interesting to decide the supremum value, $\sup \beta_0$, of such 
$\beta_0$ and study the equation in the limiting case.
For equation~\eqref{ggq-I10}, $\sup {\beta}_0 = 1$.

The second purpose of the current paper is to extend 
Theorems~\ref{ggq-th-I10} and \ref{ggq-th-I20} to a more general setting.

Let  $K \leq n$ be a fixed positive integer. Set 
\[ \fI_K = \{(i_1, \ldots, i_K): 1 \leq i_1 < \cdots < i_K \leq n, \; i_j \in \bfN\} \]
and for $I = (i_1, \ldots, i_K) \in \fI_K$ define 
\begin{equation}
\label{ggq-I550}
\Lambda_I (\lambda) 
    = \sum_{i \in I} \lambda_i = \lambda_{i_1} + \cdots + \lambda_{i_K}, 
    \;\; \lambda =  (\lambda_{1},  \ldots, \lambda_{n}) \in \bfR^n. 
\end{equation}

For convenience we fix an order for the elements in $\fI_K$:
\[ I_1, \ldots, I_N, \; \mbox{where} \; N = \frac{n!}{K! (n-K)!} \]
and write
\[ \Lambda (\lambda) = (\Lambda_{I_1} (\lambda), \ldots, \Lambda_{I_N} (\lambda)) \]
or simply,
\[ \Lambda (\lambda) = (\Lambda_{1} (\lambda), \ldots, \Lambda_{N} (\lambda)). \]
We consider the equation
\begin{equation}
\label{ggq-I10Lambda} 
  f (\Lambda (\sqrt{-1} \partial \bpartial u + X [u]))
      = \psi [u],  
\end{equation}
where 
\[ \Lambda (\sqrt{-1} \partial \bpartial u + X [u]) = \Lambda (\lambda(\sqrt{-1} \partial \bpartial u + X [u])) \]
and 
\[ \lambda(\sqrt{-1} \partial \bpartial u + X [u]) = (\lambda_{1},  \ldots, \lambda_{n}) \]
denotes the eigenvalues of $\sqrt{-1} \partial \bpartial u + X [u]$ with respect to 
$\omega$, 
$X [u] = X (z, u, \partial u, \bpartial u)$
 is a real $(1,1)$ form depending on $u$, $\partial u$ and $\bpartial u$, and 
 $f$ is a symmetric concave function defined in an open symmetric convex cone $\Gamma \subset \bfR^N$ with vertex at the origin, $\Gamma_N \subset \Gamma$, 
 and satisfies
 \begin{equation}
\label{3I-20w-Lambda}
f_i \equiv f_{\Lambda_i} \equiv 
    \frac{\partial f}{\partial \Lambda_i} \geq 0 \;\;
\mbox{in $\Gamma$}, \;\; 1 \leq i \leq N.
\end{equation}
We may similarly define admissible functions by requiring $\Lambda (\sqrt{-1} \partial \bpartial u + X [u]) \in \Gamma$. 
Clearly, equations~\eqref{ggq-I20} and ~\eqref{ggq-I10Lambda} coincide 
when $K = n-1$.

 Equation~\eqref{ggq-I10Lambda} is closely related to the general notion of $\bfG$-plurisubharmonicity of Harvey-Lawson 
 introduced and studied in a long series of papers (see e.g. \cite{HL11, HL12, HL13}).
 As in \cite{GN} we denote
\[ \rho_k (\lambda) := \prod_{1 \leq i_1 < \cdots < i_k \leq n}
(\lambda_{i_1} + \cdots + \lambda_{i_k}),   \; 1 \leq k \leq n \]
defined  in the cone
\[ \mathcal{P}_k : = \{\lambda \in \bfR^n:
      \lambda_{i_1} + \cdots + \lambda_{i_k} > 0,
         \; \forall \; 1 \leq i_1 < \cdots < i_k \leq n \}. \]
A function $u$ is $k$-plurisubharmonic 
in the sense of Harvey-Lawson~\cite{HL12} if 
\[ \lambda (\sqrt{-1} \partial \bpartial u + X) \in \cP_k.     \]

For $f (\lambda) = \log \rho_k (\lambda)$ which is concave in $\mathcal{P}_k$, the rank of $\cC_{\sigma}^ +$ is $k$.  
 Clearly, 
 \[ \rho_1(\lambda) = \sigma_n (\lambda), \;\; \rho_n  (\lambda) = \sigma_1 (\lambda), \]
 and equation~\eqref{ggq-I10} 
is equivalent to 
\[ \rho_{n-1} (\lambda (\sqrt{-1} \partial \bpartial u + X)) = \psi. \]
Note that  $\rho_K (\lambda) = \sigma_N (\Lambda (\lambda))$. So the equation
\[ \rho_K (\lambda (-\partial \bpartial u + \chi)) = \psi  \]
for $1 \leq K \leq n$ is included in the examples of equation~\eqref{ggq-I10Lambda} given by
\[ f (\Lambda) = \sigma_k^{1/k} (\Lambda), \; 1 \leq k \leq N. \]

Theorem~\ref{ggq-th-I10} holds for  
equation~\eqref{ggq-I10Lambda} with the modification of replacing 
the rank condition~\eqref{ggq-I60} by
\begin{equation}
\label{ggq-I60Lambda}
\mbox{rank of} \;\; \mathcal{C}_{\sigma}^+ \geq \frac{N (n - K)}{n} + 1, 
\;\; \forall \; \sup_{\partial \Gamma} f <  \sigma < \sup_{\Gamma} f.
\end{equation}
Indeed it extends to the equation 
\begin{equation}
\label{ggq-I10LL}
  f (\Lambda (\sqrt{-1} \partial \bpartial u + X [u]) - \beta \Lambda' (\sqrt{-1} \partial \bpartial u + X [u]))
      = \psi [u],  
\end{equation}
for some $\beta > 0$, where $\Lambda' = (\Lambda_{I'_1}, \ldots, \Lambda_{I'_N})$
and $I'_j$ denotes the complement of $I_j$ in $\{1, \ldots, n\}$;
we state a more complete result for equation~\eqref{ggq-I10LL}
to include global and boundary estimates on second derivatives.  

\begin{theorem}
\label{ggq-th-I20}
Under conditions \eqref{3I-30}, \eqref{3I-20w-Lambda}, 
\eqref{ggq-I50} and \eqref{ggq-I60Lambda},
 there exists constant $\beta_0 > 1$ such that for any $\beta < \beta_0$ and 
 admissible solution 
 $u \in C^{4} (M)$ of equation~\eqref{ggq-I10LL} %
satisfying \eqref{ggq-I45}, the interior estimate \eqref{ggn-I70}
holds for any geodesic ball $B_r$ in $M$. 

Suppose in addition that $(M^n, \omega)$
is a Hermitian manifold with smooth boundary $\partial M$ 
and compact closure $\bM = M \cup \partial M$, and 
$u \in C^{4} (\bM)$. Then
\begin{equation}
\label{ggq-C2-120}
   \max_{{\bM}} |\partial \bpartial u| \leq 
      C \Big(1 + \max_{\partial M} |\partial \bpartial u|\Big)
\end{equation}
and 
\begin{equation}
\label{ggq-C2-130}
 \max_{\partial M} |\nabla^2 u| \leq C.
\end{equation}
\end{theorem}

One may also derive interior gradient estimates under suitable growth conditions of $X$ and $\psi$ on $u$ and $\nabla u$; see Theorem~\ref{ggq-th-G10} for the precise statement.

We now turn our attentions to existence of admissible solutions 
for equation~\eqref{ggq-I10Lambda}. 
As we shall see from the proofs the results extend to equation~\eqref{ggq-I10LL} 
for $\beta < \beta_0$ (as in Theorem~~\ref{ggq-th-I20}) without difficulty.
For simplicity we only state our results here for $X$ and $\psi$ both independent of $u$ and $\nabla u$; more general results are stated in Section~\ref{ggq-E}. 
We first consider the Dirichlet problem. 

\begin{theorem}
\label{ggq-th-I30}
Let $(M^n, \omega)$ be a Hermitian manifold with smooth boundary 
$\partial M$ and compact closure $\bM = M \cup \partial M$, 
$X$ a smooth real $(1,1)$ form on $\bM$,  $\psi \in C^{\infty} (\bM)$ and $\varphi \in C^0 (\partial M)$.  
Suppose \eqref{3I-30}, \eqref{3I-40}, \eqref{3I-20w-Lambda}, 
\eqref{ggq-I50} and 
\eqref{ggq-I60Lambda} hold, and that there exists an admissible viscosity subsolution $\ul u \in C^0 (\bM)$ satisfying 
\begin{equation}
\label{ggq-I10Ls}
 \left\{
\begin{aligned}
 f (\Lambda (\sqrt{-1} \partial \bpartial \ul u + X)) \geq \,& \psi   \; \mbox{in $\bM$}, \\
\ul u = \varphi \; \mbox{on $\partial M$}.
\end{aligned} \right.
\end{equation}
Equation~\eqref{ggq-I10Lambda} 
admits a unique admissible solution $u \in C^{\infty} (M) \cap C^0 (\bM)$ with $u =\varphi$ on $\partial M$ and,  for any subdomain $\Omega \subset \ol{\Omega} \subset M$,
\begin{equation}
\label{ggq-I150}
|u|_{C^{2, \alpha} (\Omega)} \leq C, \; 0 < \alpha \leq 1
\end{equation}
where $C$ depends on the distance from $\Omega$ to $\partial M$. 

If in addition $\varphi \in C^{\infty} (\partial M)$ then $u \in C^{\infty} (\bM)$
and satisfies the global estimate
\begin{equation}
\label{ggq-I150'}
|u|_{C^{2, \alpha} (\bM)} \leq C, \; 0 < \alpha \leq 1.
\end{equation}
\end{theorem}

More precisely,  for a domain $\Omega$ in $M$ we have 
\begin{equation}
\label{ggq-I150.1}
|\nabla u|^2 + |\nabla^2 u| \leq \frac{C}{d^2} \Big\{1 + \sup_M u - u\Big\} \;\; 
\mbox{in $\Omega$}
\end{equation}
and 
\begin{equation}
\label{ggq-I155.2}
 |\nabla^2 u|_{C^{\alpha} (\Omega)} \leq \frac{C}{d^{2+ \alpha}} 
 \Big\{1 + \sup_M u - \inf_{\Omega} u\Big\}, \; 0 < \alpha \leq 1
 \end{equation}
where $d$ is the distance from $\Omega$ to $\partial M$, and $C$ is independent of $d$.

An interesting question would be when there exist admissible subsolutions. Under suitable conditions on $f$, one should be able to follow the construction of Caffarelli-Nirenberg-Spruck~\cite{CNS3} with appropriate modifications, for instance, on manifolds with
$K$-plurisubharmonic exhausting functions.

Equation~\eqref{ggq-I10Lambda} belongs to the more general class of 
equations of the form
\begin{equation}
\label{ggq-I100}
  f (\lambda (\sqrt{-1} \partial \bpartial u + X [u]))
      = \psi [u].  
\end{equation}
Consequently, when $X$ and $\psi$ are both independent of $u$ and its gradient $\nabla u$,
results of Szekelyhidi~\cite{Szekelyhidi18} apply to equation~\eqref{ggq-I10Lambda} on compact Hermitian manifold (without boundary).
Under assumption~\eqref{ggq-I60Lambda}, we obtain the following existence result for equation~\eqref{ggq-I10Lambda} without the key $\cC$-subsolution condition in \cite{Szekelyhidi18} (which, in terms of the tangent cone at infinity, is equivalent to the condition introduced in \cite{Guan14} for a type I cone $\Gamma$
that there exists a function $\ul u \in C^2 (M)$ with 
$\lambda (\sqrt{-1} \partial \bpartial \ul u + X) \in \cC^+_{\psi}$ on $M$ \cite{Guan}; recall from \cite{CNS3} that $\Gamma$ is of type I if each positive $\lambda_i$ axis is contained in $\partial \Gamma$).

\begin{theorem}
\label{ggq-th-I40}
Let $(M^n, \omega)$ be a compact Hermitian manifold, $X$ a smooth positive $(1,1)$ form and $\psi \in C^{\infty} (M)$.
Suppose \eqref{3I-30}, \eqref{3I-40}, \eqref{3I-20w-Lambda}, 
\eqref{ggq-I50} and 
\eqref{ggq-I60Lambda} hold; without loss of generality (by adding a constant if necessary) we may assume
in place of  \eqref{3I-40} that
\begin{equation}
\label{3I-40'}
 \sup_{\partial \Gamma} f \leq 0 < \psi. 
\end{equation}
Assume in addition that
\begin{equation}
\label{ggq-I45infty}
\sup_{\Gamma} f = 
\lim_{t \rightarrow +\infty} f (t {\bf{1}}) 
 > \sup_M \psi. 
\end{equation}
Then there exists a unique constant $b$ such that the equation 
\begin{equation}
\label{ggq-I10LC}
  f (\Lambda (\sqrt{-1} \partial \bpartial u + X))
      = e^b \psi
\end{equation}
has a unique admissible solution $u \in C^{\infty} (M)$ with 
\[ \sup_M u = 0. \]
\end{theorem}

Without the assumption $X > 0$ equation~\eqref{ggq-I10LC'} may become degenerate. Note that when $X > 0$ constants are admissible functions. On the other hand, if $X \leq 0$ 
there are no admissible functions on compact manifolds (without boundary).

Comparing to equation~\eqref{ggq-I10Lambda}, the general equation~\eqref{ggq-I100} is much more difficult, and most of the results in this paper no longer hold without additional assumptions even in the case $X$ and $\psi$ are independent of $u$ and $\nabla u$. We also note that, among others, equation~\eqref{ggq-I100} will only be degenerate elliptic if condition \eqref{3I-20} is replaced by \eqref{3I-20w}.
For $X [u] = X (\partial u, \bpartial u)$ and $\psi [u] = \psi (\partial u, \bpartial u)$, 
Guan-Nie \cite{GN} derived the second order estimates for equation~\eqref{ggq-I100}  on closed Hermitian manifolds under rather strong conditions; we refer the reader to 
\cite{GN} for more references. 

 For $f = (\sigma_n/\sigma_k)^{\frac{1}{n-k}}$,
the Dirichlet problem for equation~\eqref{ggq-I20}  was recently studied by
Feng-Ge-Zheng~\cite{FGZ}.
In~\cite{GQY19}, Guan-Qiu-Yuan treated equations of the form
\begin{equation}
\label{ggq-GQY}
  F (\gamma \Delta u \omega - \beta \sqrt{-1} \partial \bpartial u +  \chi [u])
      = \psi  [u] 
\end{equation}
for $\gamma > 0$, $\beta < \gamma$. One should be able to extend most of results described in the current paper to the limiting case $\beta = \gamma$ for equation~\eqref{ggq-GQY}.
It also seems interesting to study equation~\eqref{ggq-I10Lambda} with 
$\Lambda_I$ defined in \eqref{ggq-I550} replaced by other symmetric functions of $K$ variables.

The rest of this paper is organized as follows. In Section~\ref{ggq-P} 
we prove some key properties of the tangent cone at infinity related to its rank. 
The second derivative and gradient estimates are derived in Sections~\ref{ggq-C2}
and \ref{ggq-G}, respectively. We also establish the Harnack inequality using the gradient estimates for positive solutions when $X$ and $\psi$ are independent of $u$ and $\nabla u$. 
Section~\ref{ggq-E} concerns the existence of admissible solutions, where we prove  Theorems~\ref{ggq-th-I30}, \ref{ggq-th-I40} and their extensions. 

The second author wishes to thank Duong Phong for stimulating communications, especially for his comments on results described in \cite{GQY19}, which motivated us to carry out the current work.

\bigskip

\section{Rank of tangent cone at infinity}
\label{ggq-P}
\setcounter{equation}{0}

\medskip

In this section $f$ is assumed to satisfy \eqref{3I-30}, \eqref{3I-20w} 
and \eqref{ggq-I50}. 
We shall prove the following key lemmas. 

\begin{lemma}
\label{ggq-lemma-P20}
Let $P = \{\mu \in \bfR^n: \nu \cdot \mu = c\}$ be a hyperplane,  where $\nu$ is a unit vector. Suppose that 
there exists a sequence 
$\{\lambda_k\}$ in $\partial \Gamma^{\sigma}$ with 
\begin{equation}
\label{ggq-P90'}
\lim_{k \rightarrow + \infty} \nu_{\lambda_k} = \nu, \;\;
\lim_{k \rightarrow + \infty} \nu_{\lambda_k} \cdot \lambda_k = c,  \;\;
\lim_{k \rightarrow + \infty} |\lambda_k| = + \infty.  
\end{equation}
Then $P$ is a supporting hyperplane to $\mathcal{C}_{\sigma}^+$ at a non-vertex point.
\end{lemma}

\begin{proof} 
Assume $\mu_0 \in \bfR^n$ with $\nu \cdot \mu_0 \leq c - \epsilon$ 
for some $\epsilon > 0$. Then
\[ \lim_{k \rightarrow + \infty} \nu_{\lambda_k} \cdot (\mu_0 - \lambda_k) 
= \nu \cdot \mu_0 - c \leq - \epsilon. \]
This means that $S_{\mu_0}^{\sigma}$ is noncompact and therefore
 $\mu_0 \notin \mathcal{C}_{\sigma}^+$. As $\epsilon > 0$ is arbitrary, 
 we see that $\mathcal{C}_{\sigma}^+$ is contained in the half space
 $P^+ = \{\mu \in \bfR^n: \nu \cdot \mu > c\}$. 

By symmetry the vertex of $\mathcal{C}_{\sigma}^+$ is at a point
$\va = (a, \ldots, a)$ for some real number $a$. 
It is easy to see that $\va \in P$, i.e.
\[ a = \frac{c}{\sum \nu^i} \]
 where 
 $\nu = (\nu^1, \ldots, \nu^n)$. Indeed, clearly $a \geq c (\sum \nu^i)^{-1}$
 for otherwise $\va \notin \ol{P^+}$. Let $R_k$ be the ray from $\va$ passing 
 $\lambda_k$ so $R_k$ lies in $\mathcal{C}_{\sigma}^+ \subset P^+$. On the other hand, $R_k$ would not completely lie in $\ol{P^+}$ for $k$ large if $a > c (\sum \nu^i)^{-1}$, a contradiction. 

Let $\mu_k$ be the point of intersection of $R_k$ and the unit sphere centered at 
$\va$. 
Passing to a subsequence we may assume 
$\mu_k \rightarrow \mu \in \ol{\mathcal{C}_{\sigma}^+}$ as 
$k \rightarrow + \infty$. Clearly,
\[ 0 \geq \nu_k \cdot (\mu_k - \lambda_k) \geq \nu_k \cdot (\va - \lambda_k). \]
Taking limit as $k \rightarrow + \infty$ we obtain $\nu \cdot \mu = c$. 
This show that $\mu \in P \cap \partial \mathcal{C}_{\sigma}^+$.
\end{proof}

\begin{lemma}
\label{ggq-lemma-P30} 
Let $r \geq 1$ be the rank of $\mathcal{C}_{\sigma}^+$. There exists $c_1 > 0$ 
such that at any point $\lambda \in \partial \Gamma^{\sigma}$ where without loss of generality we assume $f_1 \leq \cdots \leq f_n$, 
\[ \sum_{i \leq n-r+1} f_i \geq c_1 \sum f_i. \]
\end{lemma}

\begin{proof}
We prove by contradiction. Suppose 
that for each $k = 1, 2, \ldots$ there exists 
$\lambda_k \in \partial \Gamma^{\sigma}$ with 
 $f_1 (\lambda_k) \leq \cdots \leq f_n (\lambda_k)$ and 
\begin{equation}
\label{ggq-P100}
\sum_{i \leq n-r+1} f_i (\lambda_k) \leq \frac{1}{k} \sum f_i (\lambda_k). 
\end{equation}
Note that there is $t > 0$ such that $t {\bf{1}} \in \partial \Gamma^{\sigma}$, that is,
$f (t {\bf{1}}) = \sigma$. By the concavity of $f$,
\[ \sum f_i (t - \lambda_i) \geq f (t {\bf{1}}) - f (\lambda) = 0. \]
Therefore 
\[ \sum f_i \lambda_i  \leq t \sum f_i . \]
By \eqref{ggq-I50}, 
we can pass to a subsequence and assume that
\[ \lim_{k \rightarrow + \infty} \nu_{\lambda_k} = \nu, \;\;
\lim_{k \rightarrow + \infty} \nu_{\lambda_k} \cdot \lambda_k = c,  \;\;
\lim_{k \rightarrow + \infty} |\lambda_k| = + \infty.  \]
According to Lemma~\ref{ggq-lemma-P20}, 
the hyperplane 
$P = \{\mu \in \bfR^n: \nu \cdot \mu = c\}$ is a supporting hyperplane to 
$\mathcal{C}_{\sigma}^+$. However, $\nu$ has at most $r-1$ nonzero components,
which is a contradiction. 
\end{proof}

\begin{lemma}
\label{ggq-lemma-P40}
Let  
$P = \{\mu \in \bfR^n: \nu \cdot \mu = c\}$,  where $\nu$ is a unit vector, be a 
tangent hyperplane to 
$\partial \mathcal{C}_{\sigma}^+$.  Then
there exists a sequence 
$\{\lambda_k\}$ in $\partial \Gamma^{\sigma}$ with 
\begin{equation}
\label{ggq-P90}
\lim_{k \rightarrow + \infty} \nu_{\lambda_k} = \nu, \;\;
\lim_{k \rightarrow + \infty} \nu_{\lambda_k} \cdot \lambda_k = c,  \;\;
\lim_{k \rightarrow + \infty} |\lambda_k| = + \infty.  
\end{equation}
\end{lemma}

\begin{proof}
We may assume 
$\partial \Gamma^{\sigma} \cap \partial \mathcal{C}_{\sigma}^+ \cap P = \emptyset$
for otherwise it contains a ray (see \cite{Guan14}) and 
\eqref{ggq-P90} holds for some 
$\lambda_k \in  \partial \Gamma^{\sigma} \cap \partial \mathcal{C}_{\sigma}^+ \cap P$.

Suppose that $P$ is the tangent hyperplane to  $\partial \mathcal{C}_{\sigma}^+$
at $\mu \in P \cap \partial \mathcal{C}_{\sigma}^+$.
We have $\mu \notin \partial \Gamma^{\sigma}$ and 
$\mu_k = \mu + \frac{\nu}{k} \in \mathcal{C}_{\sigma}^+$ for all $k \geq 1$  
as $\nu \in \ol{\Gamma_n}$ and $\mu + \Gamma_n \in \mathcal{C}_{\sigma}^+$.
Let
\[ S_{\mu_k}^{\sigma, 0} 
  = \{\lambda \in \partial \Gamma^{\sigma}: \nu_{\lambda} \cdot (\mu_k - \lambda) = 0\} \]
and choose $\lambda_k \in S_{\mu_k}^{\sigma, 0}$ satisfying
\[ \nu \cdot (\lambda_k - \mu) = \min_{\lambda \in S_{\mu_k}^{\sigma, 0}}
   \nu \cdot (\lambda - \mu) \geq 0; \]
such $\lambda_k$ exists as $S_{\mu_k}^{\sigma, 0}$
is compact. 
As 
 $\partial \Gamma^{\sigma} \cap \partial \mathcal{C}_{\sigma}^+ \cap P = \emptyset$,
 we have 
\[  \lim_{k \rightarrow + \infty} |\lambda_k| = + \infty. \]
 So the last limit in \eqref{ggq-P90} holds. 
Next, 
\[ \nu_{\lambda_k} \cdot (\lambda_k - \mu) 
   =  \nu_{\lambda_k} \cdot (\lambda_k - \mu_k) 
+  \nu_{\lambda_k} \cdot (\mu_k - \mu) 
=  \frac{\nu_k \cdot \nu}{k}. \]
 Therefore 
\begin{equation}
\label{ggq-P98} 
\lim_{k \rightarrow + \infty} \nu_{\lambda_k}  \cdot (\lambda_k - \mu)  = 0. 
\end{equation}
 Passing to a subsequence we may assume 
 \[  \lim_{k \rightarrow + \infty} \nu_{\lambda_k} = \nu'. \]
By \eqref{ggq-P98} we have
\[  \lim_{k \rightarrow + \infty} \nu_{\lambda_k} \cdot \lambda_k 
   = \nu' \cdot \mu \equiv c'. \]
By Lemma~\ref{ggq-lemma-P20} we see that 
$P' = \{\eta \in \bfR^n: \nu' \cdot \eta = c'\}$ is a supporting plane 
to $\mathcal{C}_{\sigma}^+$ at $\mu$. By uniqueness of the tangent plane, $P' = P$ and therefore
$\nu' = \nu$. 
 \end{proof}

\bigskip

\section{The second order estimates}
\label{ggq-C2}
\setcounter{equation}{0}
\medskip

In this section we establish the second order estimates in Theorems~\ref{ggq-th-I10} and \ref{ggq-th-I20} for admissible solutions.
We first fix some notation.
In local coordinates $z = (z_{1}, \ldots,z_{n})$, we shall write 
\begin{equation}
\label{gqy-pre213}
\omega = \sqrt{-1} g_{i\bj} dz_i \wedge dz_{\bj}
\end{equation}
and $\{g^{i\bj}\} = \{g_{i\bj}\}^{-1}$.
The Christoffel symbols $\Gamma_{ij}^k$,
torsion and curvature tensors are given respectively by
\[ \nabla_{\frac{\partial}{\partial z_i}} \frac{\partial }{\partial z_j}
= \Gamma_{ij}^k \frac{\partial}{\partial z_k}, 
\;\; T^k_{ij} = \Gamma_{ij}^k  - \Gamma_{ji}^k \]
and
\[  R_{i\bj k\bl} = - g_{m\bl}  \frac{\partial \Gamma_{ik}^m}{\partial \bz_j}
  = - \frac{\partial g_{k\bl}}{\partial z_i \partial \bz_j}
      +   g^{p\bq} \frac{\partial g_{k\bq}}{\partial z_i} \frac{\partial g_{p\bl}}{\partial \bz_j}
      \]
where $\nabla$ is the Chern connection on $(M^n, \omega)$.

Let $u \in C^{4} (M)$ be an admissible solution of 
equation~\eqref{ggq-I10Lambda}. 
We denote
\[ \fg [u] = \sqrt{-1} \partial \bpartial u + X [u] \]
and define $G$ by
\[ G (\fg [u]) = f (\Lambda (\fg [u])). \]
Consequently,  equation~\eqref{ggq-I10Lambda} takes the form
\begin{equation}
\label{ggq-I20gg}
G (\fg [u]) = \psi [u].
\end{equation}

When $K= n-1$, we have $N = n$ and 
\[ \Lambda (\fg [u]) = \lambda ((\tr \fg [u]) \omega - \fg [u]). \]
We see that equation~\eqref{ggq-I20g} is covered by equation~\eqref{ggq-I20gg}
for 
\[ X [u] = \frac{\tr \chi [u]}{n-1} \omega - \chi [u]. \]

In local coordinates, we write
\[ \fg [u] = \sqrt{-1} \fg_{i\bj} dz_i \wedge dz_{\bj} 
      = \sqrt{-1} (\partial_i \bpartial_j u + X_{i\bj}) dz_i \wedge dz_{\bj} \]
and differentiate equation~\eqref{ggq-I20gg}
 twice to obtain
\begin{equation}
\label{gblq-C65'}
G^{i\bj} \nabla_p \fg_{i\bj} = \nabla_p \psi [u]
\end{equation}
and 
\begin{equation}
\label{gblq-C75'}
  G^{i\bj} \nabla_{\bq} \nabla_p \fg_{i\bj} 
   =  \nabla_{\bq} \nabla_p \psi  [u]
    - G^{i\bj, k\bl} \nabla_p \fg_{i\bj}  \nabla_{\bq} \fg_{k\bl},
\end{equation}
where
\[ G^{i\bj} = \frac{\partial G}{\partial \fg_{i\bj}} (\fg [u]), \;\; 
    G^{i\bj, k\bl} 
     = \frac{\partial^2 G}{\partial \fg_{i\bj} \partial \fg_{k\bl}} (\fg [u]). \]

To estimate $|\partial \bpartial u|$, as in \cite{GN} we follow ideas of Tosatti-Weinkove~\cite{TW2017} and consider the quantity which is given in local coordinates by
\begin{equation}
\label{gblq-C10A}
 A := \sup_{z \in M} \max_{\xi \in T^{1,0}_z M}
  \; e^{(1 + \gamma) \phi} \fg_{p\bq} \xi_p \bar{\xi_q}
   (g^{k\bl} \fg_{i\bl} \fg_{k\bj} \xi_i \bar{\xi_j})^{\frac{\gamma}{2}}/|\xi|^{2 + \gamma}
\end{equation}
where $\phi$ is a function depending on $u$ and $|\nabla u|$, and $\gamma > 0$ is a small constant to be determined.

Assume that $A$ is achieved at an interior point $z_0 \in M$ for some 
$\xi \in T^{1,0}_{z_0} M$.
We choose local coordinates around $z_0$  such that
$g_{i\bj} = \delta_{ij}$ and $T_{ij}^k = 2 \Gamma_{ij}^k$ 
using the lemma of Streets and Tian~\cite{ST11}, and that $\fg_{i\bj}$ is diagonal 
at $z_0$ with
\[ \fg_{1\bar 1} \geq \fg_{2 \bar 2} \geq \cdots \geq \fg_{n\bn}. \]
As noted in \cite{TW2017} the maximum $A$ is achieved for $\xi = \partial_1$ at $z_0$ when $\gamma > 0$ is sufficiently small; see also \cite{GN}.
We assume $\fg_{1\bar{1}} \geq 1$;  otherwise we are done.  


Let $W = g_{1\bar{1}}^{-1} g^{k\bl} \fg_{1\bl} \fg_{k\bar{1}}$.
We see that the function
$e^{(1+\gamma)  \phi} g_{1\bar{1}}^{-1} \fg_{1\bar{1}} W^{\frac{\gamma}{2}}$ which is  locally well defined
attains a maximum $A = (e^{\phi} \fg_{1\bar{1}})^{1+\gamma}$ at $z_0$ where
$W = \fg_{1\bar{1}}^2$,
\begin{equation}
\label{gblq-C80'}
 \begin{aligned}
\frac{\partial_i (g_{1\bar{1}}^{-1} \fg_{1\bar{1}})}{\fg_{1\bar{1}}}
   +  \frac{\gamma \partial_i W}{2 W} + (1+ \gamma)  \partial_i \phi = \,& 0, \\
\frac{\bpartial_i (g_{1\bar{1}}^{-1} \fg_{1\bar{1}})}{\fg_{1\bar{1}}}
   +\frac{ \gamma \bpartial_i W}{2 W} + (1+ \gamma) \bpartial_i \phi = \,& 0
\end{aligned}
\end{equation}
for each $1 \leq i \leq n$, and
\begin{equation}
\label{gblq-C90'}
\begin{aligned}
0 \geq \,&
   \frac{1}{\fg_{1\bar{1}}} G^{i\bi} \bpartial_i \partial_i (g_{1\bar{1}}^{-1} \fg_{1\bar{1}})
   - \frac{1}{\fg_{1\bar{1}}^2} G^{i\bi}
     \partial_i (g_{1\bar{1}}^{-1} \fg_{1\bar{1}}) \bpartial_i (g_{1\bar{1}}^{-1} \fg_{1\bar{1}}) \\
  & + \frac{\gamma}{2 W} G^{i\bi} \bpartial_i \partial_i W
     - \frac{\gamma}{2 W^2} G^{i\bi} \partial_i W \bpartial_i W
   + (1 + \gamma) G^{i\bi} \bpartial_i \partial_i \phi.
\end{aligned}
\end{equation}

In what follows we shall make use of calculations in \cite{GN}; for the reader's convenience we give a brief outline. First,
\begin{equation}
\label{gn-S115}
\begin{aligned}
\partial_i (g_{1\bar{1}}^{-1} \fg_{1\bar{1}}) = \nabla_i \fg_{1\bar{1}}, \;\;
 \partial_i W 
 = 2 \fg_{1\bar{1}} \nabla_i \fg_{1\bar{1}},
     \end{aligned}
      \end{equation}
\begin{equation}
\label{gn-S108}
 \begin{aligned}
\bpartial_j  \partial_i (g_{1\bar{1}}^{-1} \fg_{1\bar{1}})
       = \,& \nabla_{\bj} \nabla_i \fg_{1\bar{1}}
    + (\ol{\Gamma_{j1}^m} \nabla_i \fg_{1\bm}
        - \ol{\Gamma_{j1}^1} \nabla_i \fg_{1\bar{1}}) \\
      & + (\Gamma_{i1}^m \nabla_{\bj} \fg_{m\bar{1}}
            - \Gamma_{i1}^1 \nabla_{\bj} \fg_{1\bar{1}})
      + (\Gamma_{i1}^1  \ol{\Gamma_{j1}^1} - \Gamma_{i1}^m \ol{\Gamma_{j1}^m}) \fg_{1\bar{1}}.
     \end{aligned}  
 \end{equation}     
      and
\begin{equation}
\label{gn-S110}
 \begin{aligned}
\bpartial_j \partial_i W
   = \,& 2 \fg_{1\bar{1}} \nabla_{\bj} \nabla_i \fg_{1\bar{1}}
   + 2 \nabla_i \fg_{1\bar{1}} \nabla_{\bj} \fg_{1\bar{1}}
   + \sum_{l>1} \nabla_i \fg_{l\bar{1}} \nabla_{\bj} \fg_{1\bl}    \\
     &  + \sum_{l > 1} (\nabla_i \fg_{1\bl} + {\Gamma_{i1}^l} \fg_{l\bl})
       (\nabla_{\bj} \fg_{l\bar{1}} + \ol{\Gamma_{j1}^l} \fg_{l\bl}) \\
   & + \fg_{1\bar{1}} \sum_{l > 1} (\ol{\Gamma_{j1}^l} \nabla_i \fg_{1\bl}
       + \Gamma_{i1}^l  \nabla_{\bj} \fg_{l\bar{1}}) \\
  &    -  \fg_{1\bar{1}} \sum_{l > 1} \Gamma_{i1}^m \ol{\Gamma_{j1}^m} (\fg_{1\bar{1}} +  \fg_{l\bl)}.
              \end{aligned}
 \end{equation}
It follows that
\begin{equation}
\label{gblq-C103}
\begin{aligned}
G^{i\bi} \partial_i W \bpartial_i W
= 4 \fg_{1\bar{1}}^2 G^{i\bi} \nabla_i \fg_{1\bar{1}} \nabla_{\bi} \fg_{1\bar{1}},
\end{aligned}
\end{equation}
\begin{equation}
\label{gblq-C105}
\begin{aligned}
G^{i\bi} \partial_i (g_{1\bar{1}}^{-1} \fg_{1\bar{1}}) \bpartial_i (g_{1\bar{1}}^{-1} \fg_{1\bar{1}})
= G^{i\bi} \nabla_i \fg_{1\bar{1}} \nabla_{\bi} \fg_{1\bar{1}}
\end{aligned}
\end{equation}
and, by Cauchy-Schwarz inequality,
\begin{equation}
\label{gblq-C100}
 \begin{aligned}
G^{i\bi} \bpartial_i \partial_i W
\geq \,& 2 \fg_{1\bar{1}} G^{i\bi} \nabla_{\bi} \nabla_i \fg_{1\bar{1}}
    + 2 G^{i\bi}  \nabla_i \fg_{1\bar{1}} \nabla_{\bi} \fg_{1\bar{1}} \\
  + \,& \sum_{l > 1} G^{i\bi} \nabla_i \fg_{1\bl} \nabla_{\bi} \fg_{l\bar{1}}
  + \frac{1}{2} \sum_{l > 1} G^{i\bi} \nabla_i \fg_{1\bl} \nabla_{\bi} \fg_{l\bar{1}}
   - C \fg_{1\bar{1}}^2 \sum G^{i\bi},
 \end{aligned}
 \end{equation}
\begin{equation}
\label{gblq-C101}
 \begin{aligned}
G^{i\bi} \bpartial_i  \partial_i (g_{1\bar{1}}^{-1} \fg_{1\bar{1}})
\geq \,& G^{i\bi} \nabla_{\bi} \nabla_i \fg_{1\bar{1}}
  - \frac{\gamma}{8 \fg_{1\bar{1}}} \sum_{l > 1} G^{i\bi}
            \nabla_i \fg_{1\bl} \nabla_{\bi} \fg_{l\bar{1}}
        - C \fg_{1\bar{1}} \sum G^{i\bi}.
 \end{aligned}
 \end{equation}
 We derive from \eqref{gblq-C80'}, \eqref{gblq-C90'} and 
 \eqref{gblq-C103}-\eqref{gblq-C101} that
  \begin{equation}
\label{gblq-C80''}
 \begin{aligned}
\nabla_i \fg_{1\bar{1}} + \fg_{1\bar{1}} \partial_i \phi = 0, \;
\nabla_{\bi} \fg_{1\bar{1}} + \fg_{1\bar{1}} \bpartial_i \phi = 0
\end{aligned}
\end{equation}
and
 \begin{equation}
\label{gblq-C91}
\begin{aligned}
0  \geq \,& \frac{1}{\fg_{1\bar{1}}}  G^{i\bi} \nabla_{\bi} \nabla_i \fg_{1\bar{1}}
 - \frac{1}{\fg_{1\bar{1}}^2} G^{i\bi}  \nabla_i \fg_{1\bar{1}} \nabla_{\bi} \fg_{1\bar{1}}
   + G^{i\bi} \bpartial_i \partial_i \phi \\
 & + \frac{\gamma}{\fg_{1\bar{1}}^2} \sum_{l > 1} G^{i\bi} \nabla_i \fg_{1\bl} \nabla_{\bi} \fg_{l\bar{1}} + \frac{\gamma}{16 \fg_{1\bar{1}}^2} \sum_{l > 1} G^{i\bi}
            \nabla_i \fg_{1\bl} \nabla_{\bi} \fg_{l\bar{1}}
        - C \sum G^{i\bi}.
\end{aligned}
\end{equation}

As in \cite{GN} and \cite{GQY19} we have
\begin{equation}
\label{gblq-R155}
 \begin{aligned}
 \nabla_{\bi} \nabla_i \fg_{1\bar{1}}  -  \nabla_{\bar{1}} \nabla_1 \fg_{i\bi}
   = \,&  R_{i\bi 1\bar{1}} \fg_{1\bar{1}} - R_{1\bar{1} i\bi} \fg_{i\bi}
         - T_{i1}^l \nabla_{\bi} \fg_{l\bar{1}}  - \ol{T_{i1}^l} \nabla_i \fg_{1\bl} \\
     &  - T_{i1}^l  \ol{T_{i1}^l} \fg_{l\bl} + H_{i\bi}
  \end{aligned}
 \end{equation}
  where
\[ \begin{aligned}
    H_{i\bi} = \,& \nabla_{\bi} \nabla_i X_{1\bar{1}}
                       -  \nabla_{\bar{1}} \nabla_1 X_{i\bi}
      - 2 \fRe\{T_{i1}^l \nabla_{\bi} X_{l\bar{1}}\} 
      + R_{i\bi 1\bl} X_{l\bar{1}} - R_{1\bar{1} i\bl} X_{l\bi}
       - T_{i1}^j  \ol{T_{i1}^l} X_{j\bl}.
  \end{aligned} \]
 It follows from Schwarz inequality that 
\begin{equation}
\label{gblq-R156}
 \begin{aligned}
G^{i\bi} \nabla_{\bi} \nabla_i \fg_{1\bar{1}}
   \geq \,& G^{i\bi} \nabla_{\bar1} \nabla_1\fg_{i\bi}
         -  \frac{\gamma}{32 \fg_{1\bar{1}}} 
             G^{i\bi}  \nabla_i \fg_{1\bl} \nabla_{\bi} \fg_{l\bar{1}} \\
       &  - C \fg_{1\bar{1}} \sum G^{i\bi} + G^{i\bi} H_{i\bi}.
 \end{aligned}
 \end{equation}
From \eqref{gblq-C75'},  \eqref{gblq-C80''}, \eqref{gblq-C91} and \eqref{gblq-R156} 
we derive
 \begin{equation}
\label{gblq-C90}
\begin{aligned}
\fg_{1\bar{1}} G^{i\bi} \bpartial_i \partial_i \phi
   \leq  \,&  - \nabla_{\bar{1}} \nabla_1 \psi + 
                 \frac{1 + \gamma}{\fg_{1\bar{1}}} G^{i\bi}  \nabla_i \fg_{1\bar{1}} \nabla_{\bi} 
                 \fg_{1\bar{1}} - G^{i\bi} H_{i\bi} \\
      & -  \frac{\gamma}{32 \fg_{1\bar{1}}} \sum_{l > 1} G^{i\bi}
            (\nabla_i \fg_{l\bar{1}} \nabla_{\bi} \fg_{1\bl} + \nabla_i \fg_{1\bl} \nabla_{\bi} \fg_{l\bar{1}}) +  C \fg_{1\bar{1}} \sum G^{i\bi} \\
       \leq  \,&  - \nabla_{\bar{1}} \nabla_1 \psi + 
    (1+\gamma) \fg_{1\bar{1}} G^{i\bi}  \nabla_i \phi \nabla_{\bi} \phi - G^{i\bi} H_{i\bi}  +  C \fg_{1\bar{1}} \sum G^{i\bi}.
 \end{aligned}
\end{equation}

By straightforward calculations (see e.g. \cite{GN}, \cite{GQY19}),
\begin{equation}
\label{gblq-R155.4}
\begin{aligned}
G^{i\bi} H_{i \bi} 
   \geq \,&  2 G^{i\bi} \fRe\{X_{1 \bar{1}, \zeta_{\alpha}}
                     \nabla_{\alpha} \nabla_{\bi} \nabla_i   u\}
                 -  2 G^{i\bi} \fRe\{X_{i \bi, \zeta_{\alpha}} \nabla_{\alpha} 
                     \nabla_{\bar{1}} \nabla_1 u\} \\
           & - C \fg_{1\bar{1}}^2 \sum G^{i\bi}  
               - C \sum_{i, k} G^{i\bi} |\nabla_i  \nabla_{k} u|^2
               - C \sum_k |\nabla_1 \nabla_k u|^2 \sum G^{i\bi}  \\
 \geq \,& 2 \fRe\{X_{1 \bar{1}, \zeta_{\alpha}}  \nabla_{\alpha} \psi\}
          + 2 \fg_{1\bar{1}} G^{i\bi} \fRe\{X_{i \bi, \zeta_{\alpha}}  \nabla_{\alpha} \phi\} 
          - C |A|^2 \sum G^{i\bi}.
             \end{aligned}
 \end{equation}  
Henceforth for convenience we denote
\[ |A_i|^2 = \fg_{i\bi}^2 + \sum_{k} |\nabla_i \nabla_k u|^2, \;\; 
    |A|^2 =  \sum |A_i|^2. \]
Next, 
\begin{equation}
\label{gblq-R155.6}
\nabla_{\alpha} \psi = \psi_{\alpha} + \psi_u \nabla_{\alpha} u 
        + \psi_{\zeta_\beta} \partial_{\alpha} \partial_{\beta} u 
        + \psi_{\bzeta_\beta} \partial_{\alpha} \bpartial_{\beta} u 
        \end{equation}        
and 
\begin{equation}
\label{gblq-R155.5}
\begin{aligned}
 \nabla_{\bar{1}} \nabla_1 \psi 
     \geq \,& \psi_{\zeta_\alpha} \nabla_{\alpha} \nabla_{\bar{1}} \nabla_1 u
                  + \psi_{\bzeta_\alpha} \nabla_{\balpha} \nabla_{\bar{1}} \nabla_1  u 
                  -  C |A_1|^2 \\ 
     \geq \,& - \fg_{1\bar{1}} \psi_{\zeta_\alpha} \nabla_{\alpha} \phi
                  - \fg_{1\bar{1}} \psi_{\bzeta_\alpha} \nabla_{\balpha} \phi -  C |A|^2. 
\end{aligned}
\end{equation}   
Finally, from \eqref{gblq-C90}-
\eqref{gblq-R155.5}
we derive
\begin{equation}
\label{gblq-C150}
\begin{aligned}
\fg_{1\bar{1}} G^{i\bi} (\bpartial_i \partial_i \phi 
    - (1 + \gamma) \bpartial_i  \phi  \partial_i \phi)
       \leq \,& C (\fg_{1\bar{1}} |\nabla \phi| + |A|^2) 
       \Big(1 + \sum G^{i\bi}\Big). 
\end{aligned}
\end{equation}

Let 
$\phi = \log (\eta/h)$ where $h = (1 - \gamma |\nabla u|^2)$, $\gamma$ as before, and 
$\eta$ is a smooth function to be chosen; we require $\gamma$ small to satisfy 
$2 \gamma |\nabla u|^2  \leq 1$.  
By straightforward calculations,
\begin{equation}
\label{hess-a264}
\partial_i  |\nabla u|^2
    =  \nabla_k u \nabla_i \nabla_{\bk} u + \nabla_{\bk} u \nabla_i \nabla_k u
\end{equation}
and 
\begin{equation}
\label{hess-a266}
  \begin{aligned}
 \bpartial_i \partial_i |\nabla u|^2 
   = \,& \nabla_i \nabla_{\bk} u \nabla_k \nabla_{\bi} u 
            +  \nabla_i \nabla_{k} u \nabla_{\bi} \nabla_{\bk} u \\
         &   + \nabla_{\bk} u \nabla_{\bi} \nabla_i \nabla_k u
            + \nabla_{k} u \nabla_{\bi} \nabla_i \nabla_{\bk} u \\
   = \,& \nabla_i \nabla_{\bk} u \nabla_k \nabla_{\bi} u 
            +  \nabla_i \nabla_{k} u \nabla_{\bi} \nabla_{\bk} u \\
         &   + \nabla_{\bk} u \nabla_k \nabla_{\bi} \nabla_i u 
            +  \nabla_k u \nabla_{\bk} \nabla_{\bi} \nabla_i u \\
         &   + R_{i\bi k\bl} \nabla_l u \nabla_{\bk} u
            - T^{l}_{ik} \nabla_{l\bi} u \nabla_{\bk} u 
            - \ol{T^{l}_{ik}} \nabla_{i\bl} u \nabla_k  u \\
     \geq  \,& (1 - \gamma) |A_i|^2 + \nabla_{\bk} u \nabla_k \nabla_{\bi} \nabla_i u 
                   + \nabla_k u \nabla_{\bk} \nabla_{\bi} \nabla_i u    - C |\nabla u|^2
\end{aligned}
\end{equation}
It follows that
\[  G^{i\bi} \bpartial_i  h  \partial_i h 
      \leq C \gamma^2 G^{i\bi} |A_i|^2
    \]
    and
\begin{equation}
  \begin{aligned}
- G^{i\bi}  \bpartial_i  \partial_i  h
    \geq \,& \gamma (1 - \gamma)  \sum 
          G^{i\bi} |A_i|^2 
           - C |A| \sum G^{i\bi} - C |A|.
\end{aligned} 
\end{equation}

Clearly $|\nabla \phi| \leq |\nabla h|/h + |\nabla \eta|/\eta$ and 
\begin{equation}
\label{hess-a268}
  \begin{aligned}
  G^{i\bi} \bpartial_i  \phi  \partial_i \phi 
    \leq \,&  \frac{1 + \gamma}{h^{2}} G^{i\bi} \bpartial_i  h  \partial_i h 
    + \frac{1 + \gamma}{\gamma \eta^2} G^{i\bi} \bpartial_i  \eta \partial_i \eta.
    \end{aligned} 
\end{equation}
We see that
\[ \begin{aligned}
G^{i\bi} (\bpartial_i \partial_i \phi - (1 + \gamma) \bpartial_i  \phi  \partial_i \phi)
   \geq \,& - G^{i\bi} \Big(\frac{\bpartial_i \partial_i h}{h}
                +  \frac{3 \gamma \bpartial_i  h  \partial_i h}{h^2}\Big) 
                - \frac{C}{\gamma} G^{i\bi} \bpartial_i  \eta \partial_i \eta 
                + G^{i\bi} \bpartial_i \partial_i \eta\\
   \geq\,& \gamma (1 - \gamma - C \gamma^2) \sum G^{i\bi} |A_i|^2 
                -  \frac{C}{\gamma \eta^2} G^{i\bi} \bpartial_i  \eta \partial_i \eta \\
            &  + \frac{1}{\eta}G^{i\bi} \bpartial_i \partial_i \eta - C |A| \sum G^{i\bi} - C |A|.
\end{aligned} \]
We can now fix $\gamma$ sufficiently small, further requiring that 
  $2 \gamma + C \gamma^2 \leq 1$, to derive
  \begin{equation}
\label{gblq-C150a}
\begin{aligned}
\gamma^2  \fg_{1\bar{1}}
    \sum G^{i\bi}|A_i|^2 
   +  \fg_{1\bar{1}} G^{i\bi} \Big(\frac{\bpartial_i \partial_i \eta}{\eta} - \frac{C \bpartial_i  \eta \partial_i \eta}{\eta^2}\Big)
      \leq   C |A|^2 \Big(1 + \sum G^{i\bi}\Big).
\end{aligned}
\end{equation}  

From Lemma~\ref{ggq-lemma-P30} we have the following key inequality
for each $i \geq 1$,
\begin{equation}
 \label{ggq-C2-100}
 G^{i\bi} = \sum_{l=1}^N f_{\Lambda_l} \frac{\partial \Lambda_l}{\partial \lambda_i} 
               = \sum_{i \in I} f_{\Lambda_I} \geq c_1 \sum_{l=1}^N f_{\Lambda_l}, 
\end{equation}                
where the sum $\sum_{i \in I} f_{\Lambda_I}$ is taken over all $I \in \fI_K$ with $i \in I$, since 
$\lambda (\fg) = (\fg_{1\bar{1}}, \ldots, \fg_{n\bn})$. 
Note that when $\fg_{i\bj}$ is diagonal,  so is $G^{i\bj}$. 

Consequently, by \eqref{gblq-C150a} we obtain
  \begin{equation}
 \label{gblq-C150b}
\begin{aligned}
\frac{c_1 \gamma^2 }{2} \fg_{1\bar{1}} |A|^2 \sum  f_{\Lambda_l}            
   +  \fg_{1\bar{1}} G^{i\bi} \Big(\frac{\bpartial_i \partial_i \eta}{\eta} - \frac{C \bpartial_i  \eta \partial_i \eta}{\eta^2}\Big)
      \leq   C |A|^2 
\end{aligned}
\end{equation}  
provided that $ \fg_{1\bar{1}}$ is large enough. 
  
By the concavity of $f$ and Schwarz inequality, we derive
\[  \begin{aligned}
    \sqrt{\fg_{1\bar1}}  \sum  f_{\Lambda_l}  
        = \,& \sqrt{\fg_{1\bar1}} \sum  f_{\Lambda_l}  
        - \sum  f_{\Lambda_l} \Lambda_l (\fg) + \sum  f_{\Lambda_l} \Lambda_l (\fg) \\
   \geq \,&  f (\sqrt{\fg_{1\bar1}} {\bf{1}}) - f (\Lambda (\fg)) 
        - \epsilon \sum  f_{\Lambda_l} \Lambda_l^2 - \frac{1}{4 \epsilon} \sum  f_{\Lambda_l} \\
        \geq \,&  \frac{c_0}{2} -  \epsilon \sum  f_{\Lambda_l} \Lambda_l^2 - \frac{1}{4 \epsilon} \sum  f_{\Lambda_l}
  \end{aligned} \]
by assumption~\eqref{ggq-I45}, provided that $\fg_{1\bar1}$ is sufficiently large.
So from \eqref{gblq-C150b}  we obtain
 \begin{equation}
 \label{ggq-C150}
\begin{aligned}
\fg_{1\bar{1}} |A|^2 \sum f_{\Lambda_l} + \sqrt{\fg_{1\bar{1}}} |A|^2            
   +  C  \fg_{1\bar{1}} G^{i\bi} \Big(\frac{\bpartial_i \partial_i \eta}{\eta} - \frac{C \bpartial_i  \eta \partial_i \eta}{\eta^2}\Big)
      \leq   C |A|^2. 
\end{aligned}
\end{equation}

To derive the interior estimate, 
following \cite{GW03a} we take $\eta$ to be a smooth 
function with compact support in $B_{r}\subset M$
satisfying
\begin{equation}
\label{2-22}
0\leq \eta \leq 1, ~~\eta|_{B_{\frac{r}{2}}}\equiv 1,
~~|\nabla \eta | \leq \frac{C \sqrt{\eta}}{r},  
~~|\nabla^{2} \eta| \leq \frac{C}{r^2},
\end{equation}
so that 
\begin{equation}
\label{2-23}
\begin{aligned}
   G^{i\bi} \bpartial_i \partial_i \eta - \frac{C}{\eta} G^{i\bi} \bpartial_i  \eta \partial_i \eta
\leq \frac{C}{r^2} \sum f_{\Lambda_i}.
\end{aligned} 
\end{equation}
We derive a bound $\eta \fg_{1\bar{1}} \leq C/r^2$ at $z_0$ which gives \eqref{ggn-I70}.

Suppose that $(M^n, \omega)$
is a Hermitian manifold with smooth boundary $\partial M$.
Taking $\eta = 1$ we obtain the global estimate \eqref{ggq-C2-120}.
 The boundary estimate \eqref{ggq-C2-130}  
may be derived as in \cite{GQY19} 
 with some minor modifications; we shall omit the details here.

\begin{remark}
\label{ggq-remark-C2-10}
When $X$ and $\psi$ are independent of $u$ and $\nabla u$, we obtain  
\begin{equation}
\label{ggq-C200}
\sup_{B_{\frac{r}{2}}} |\bpartial \partial u| 
   \leq \frac{C}{r^2} \Big\{1 + \sup_{B_r} |\partial u|^2\Big\}. 
\end{equation}
\end{remark}

\bigskip

\section{Gradient estimates}
\label{ggq-G}
\setcounter{equation}{0}

In this section we derive gradient estimates. As usual it requires suitable growth conditions of $X$ and $\psi$ on $u$ and its gradient. Concerning $X$, we assume the following sub-linear growth assumption
 for $D_{\zeta} X$ when $|\zeta|$ is sufficiently large,
\begin{equation}
\label{ggq-G10}
 |D_{\zeta} X (z, u, \zeta, \bzeta)| \leq \varrho_0 |\zeta|, 
 \;\; D_u X \leq (\varrho_0 |\zeta|^2 + \varrho_1) \omega
 \end{equation}
where $\varrho_1 = \varrho_1 (z, u)$ and 
 $\varrho_0 = \varrho_0 (z, u, |\zeta|) \rightarrow 0^+$ as $|\zeta| \rightarrow \infty$;
clearly we may assume $t \varrho_0 (z, u, t)$ to be increasing in $t > 0$.
It follows that 
\[ |X| \leq C \varrho_0 |\zeta|^2 + \varrho_1 (z, u) \]
for some function $\varrho_1$; we shall only need (for convenience we drop the 
constant $C$)  
\begin{equation}
\label{ggq-G15}
   X \leq (\varrho_0 |\zeta|^2 + \varrho_1) \omega 
   \end{equation}
   which will also be crucial to the existence, more precisely to the upper bound, of admissible solutions of the Dirichlet problem in Theorem~\ref{ggq-th-E30} in Section~\ref{ggq-E}.
We impose a similar condition on $\psi$ but, as in \cite{GQY19}, associate it with the growth of $f$:
\begin{equation}
\label{ggq-G20}
 |D_\zeta \psi (z, u, \zeta, \bzeta)| \leq 
    \varrho_0  f (|\zeta|^{2} {\bf{1}})/|\zeta|, \;\;
    - D_u \psi  \leq  \varrho_0  f (|\zeta|^{2} {\bf{1}}) + \varrho_1 (z, u)
\end{equation}    
which in particular implies 
\begin{equation}
\label{ggq-G25}
 \psi  \leq  \varrho_0  f (|\zeta|^{2} {\bf{1}}) + \varrho_1 (z, u).
\end{equation}

\begin{theorem}
\label{ggq-th-G10}
Let $u \in C^3 (\bar{B}_r)$ be an admissible solution of 
equation~\eqref{ggq-I10Lambda}.
Under conditions  
\eqref{3I-30}, \eqref{3I-40}, \eqref{3I-20w-Lambda}, 
\eqref{ggq-I50},  
\eqref{ggq-I60Lambda},  \eqref{ggq-G10} and \eqref{ggq-G20},
$u$ satisfies the following interior gradient estimate
\begin{equation}  
\label{ggq-30}
 |\partial u|^2 \leq \frac{C}{r^2} \Big\{1 + \sup_{\bar{B}_r} u - u\Big\} \;\; \mbox{in ${B}_{\frac{r}{2}}$}
\end{equation}
for some uniform constant $C > 0$. 
\end{theorem}

\begin{remark}
\label{ggq-remark-G10}
We may replace \eqref{ggq-G10} and \eqref{ggq-G20} by the weaker
assumptions ~\eqref{ggq-G15}, \eqref{ggq-G25} and the linear-like growth conditions
\begin{equation}
\label{ggq-G10w}
 |D_{\zeta} X (z, u, \zeta, \bzeta)| \leq \varrho (z, u) |\zeta|,
 \end{equation}
\begin{equation}
\label{ggq-G20w}
 |D_\zeta \psi (z, u, \zeta, \bzeta)| \leq 
    \varrho (z, u) f (|\zeta|^{2} {\bf{1}})/|\zeta|,
\end{equation}    
where $\varrho > 0$. In this case we can derive the bound
\begin{equation}  
\label{ggq-30'}
 |\partial u|^2 \leq \frac{C}{r^2} \Big\{1 + u - \inf_{\bar{B}_r} u\Big\} ^N \;\; \mbox{in ${B}_{\frac{r}{2}}$}
\end{equation}
for sufficiently large $N > 0$; see e.g. \cite{GQY19}. 
\end{remark}

\begin{remark}
\label{ggq-remark-G20}
When $\omega$ is K\"ahler, \eqref{ggq-G10w} and \eqref{ggq-G20w} can be replaced 
by the assumptions 
 \begin{equation}
\label{ggq-G10c}
 D^2_{p_i p_j} X (z, u, p) \xi_i \xi_j \leq \varrho (z, u) |\xi|^2, \;\; \mbox{for $\xi \in T_z M$}, 
 \end{equation}
that is, $X$ is semi-concave in $\nabla u$, the real gradient of $u$, and
\begin{equation}
\label{ggq-G20c}
 D^2_{p_i p_j} \psi (z, u, p) \xi_i \xi_j   \geq 
    - \varrho (z, u) f (|p|^{2} {\bf{1}}) |\xi|^2/|p|^2, \;\; \xi \in T_z M.
\end{equation}    

\end{remark}

\begin{proof}[Proof of Theorem~\ref{ggq-th-G10}]
Consider 
\[ \max_{\bM} e^{\phi}  |\nabla u|^2 \]
where $\phi$ is a function to be chosen.
Suppose it is attained at an interior point $z_0 \in M$. 
We choose local coordinates $(z_{1}, \ldots, z_{n})$ such that 
$g_{i\bj} = \delta_{ij}$, 
$T_{ij}^k = 2 \Gamma_{ij}^k$ and $\fg_{i\bj} $ are diagonal at $z_0$.
The function 
$\log |\nabla u|^2 + \phi$ 
achieves a maximum at $z_{0}$ where we assume $|\nabla u| \geq 1$.
Therefore at $z_0$, 
\begin{equation}
\label{gqy-G32}
\partial_i |\nabla u|^2 + |\nabla u|^2 \partial_i \phi = 0, \;\;
\bpartial_i |\nabla u|^2 + |\nabla u|^2 \bpartial_i \phi = 0
\end{equation}
and 
\begin{equation}
\label{gqy-G34}
G^{i\bi} \frac{\bpartial_i \partial_i |\nabla u|^2}{|\nabla u|^2}
- G^{i\bi} \frac{\bpartial_i |\nabla u|^2 \partial_i |\nabla u|^2}{|\nabla u|^4}
   + G^{i\bi} \bpartial_i \partial_i \phi  \leq 0.
\end{equation}

By  \eqref{hess-a264}, \eqref{hess-a266} and Schwarz inequality,
\begin{equation}
\label{hess-a264'}
\bpartial_i  |\nabla u|^2 \partial_i  |\nabla u|^2
    \leq 2  |\nabla u|^2 |Q_i|^2
    \end{equation}
    and
\begin{equation}
\label{hess-a267}
  \begin{aligned}
 \bpartial_i \partial_i |\nabla u|^2 
     \geq  \,& (1 - \gamma) |Q_i|^2 + 
          \nabla_{\bk} u \nabla_k \nabla_{\bi} \nabla_i u 
                   + \nabla_k u \nabla_{\bk} \nabla_{\bi} \nabla_i u    - C |\nabla u|^2.
\end{aligned}
\end{equation}
where $0 < \gamma < 1/6$ and 
\[ |Q_i|^2 = \nabla_i \nabla_{\bk} u \nabla_k \nabla_{\bi} u 
           +  \nabla_i \nabla_{k} u \nabla_{\bi} \nabla_{\bk} u
           = \sum_k (|\nabla_i \nabla_{\bk} u|^2 + |\nabla_i \nabla_{k} u|^2). \]
It follows from \eqref{gblq-C65'} that
\begin{equation}
\label{ggq-G120}
 G^{i\bi} \nabla_k \nabla_{\bi} \nabla_i u 
    = G^{i\bi} (\nabla_k \fg_{i\bi} - \nabla_k {X}_{i\bi}) 
     = \nabla_k \psi - G^{i\bi} \nabla_k {X}_{i\bi}. 
 \end{equation}   
 So 
\begin{equation}
\label{ggq-G45}
  \begin{aligned}
 G^{i\bi} \bpartial_i \partial_i |\nabla u|^2 
     \geq  \,&  G^{i\bi} (1 - \gamma) |Q_i|^2 - C |\nabla u|^2 \sum G^{i\bi} + R
  \end{aligned}
\end{equation}
where 
\[ R = 2 \fRe\{(\nabla_k \psi - G^{i\bi} \nabla_k X_{i\bi}) \nabla_{\bk}  u\}.  \]  
We derive
\begin{equation}
\label{ggq-G130}
\begin{aligned} 
 G^{i\bi} \bpartial_i \partial_i \phi 
     - \frac{1+ \gamma}{2} G^{i\bi} \bpartial_i \phi \partial_i \phi
  \leq \,
          &  
          - \frac{R}{|\nabla u|^2}  
              + C \sum G^{i\bi}.
\end{aligned}
\end{equation}

Let 
$\phi = \log (\eta/h)$ where $h = 1 + \sup_M u - u$. So
\[ G^{i\bi} \bpartial_i \partial_i \phi 
    = \frac{1}{h} G^{i\bi} \bpartial_i \partial_i u  
    + \frac{1}{h^2} G^{i\bi} \bpartial_i u \partial_i u
   +  \frac{1}{\eta} G^{i\bi}\bpartial_i \partial_i \eta 
    - \frac{1}{\eta^2} G^{i\bi}\bpartial_i \eta \partial_i \eta. \]
As before, $|\nabla \phi| \leq |\nabla u|/h + |\nabla \eta|/\eta$ and 
\[    \begin{aligned}
  G^{i\bi} \bpartial_i  \phi  \partial_i \phi 
    \leq \,&  \frac{1 + \gamma}{h^{2}} G^{i\bi} \bpartial_i  u  \partial_i u 
    + \frac{1 + \gamma}{\gamma \eta^2} G^{i\bi} \bpartial_i  \eta \partial_i \eta.
    \end{aligned} \]
From \eqref{ggq-C2-100} it follows that 
\begin{equation}
\label{ggq-G140}
  \begin{aligned}
   \frac{1}{h^2} G^{i\bi} \bpartial_i u \partial_i u 
- \frac{1+ \gamma}{2} G^{i\bi} \bpartial_i \phi \partial_i \phi
& \geq 
                \frac{1- 3 \gamma}{2 h^2} G^{i\bi} \bpartial_i  u  \partial_i u 
     -  \frac{C}{\eta^2} G^{i\bi} \bpartial_i  \eta \partial_i \eta \\
  \,&   \geq \frac{c_1 |\nabla u|^2}{4 h^2}  \sum G^{i\bi}
    - \frac{C}{\eta^2} G^{i\bi} \bpartial_i  \eta \partial_i \eta.
      \end{aligned}  
\end{equation}

By the concavity of $f$ and assumption~\eqref{ggq-I50}, 
\[  \begin{aligned}
   |\nabla u|^2 \sum G^{i\bi}
    \geq \,& f (|\nabla u|^2 {\bf{1}}) - f (\Lambda) + G^{i\bi} \fg_{i\bi} 
   \geq  f (|\nabla u|^2 {\bf{1}}) - \psi - C \sum  G^{i\bi},
      \end{aligned} \] 
Similarly,
\[ \begin{aligned}
  G^{i\bi} \bpartial_i \partial_i u 
 = \,& G^{i\bi}  \fg_{i\bi} - G^{i\bi}  X_{i\bi} 
    \geq - G^{i\bi}  X_{i\bi} - C \sum G^{i\bi}
                 \end{aligned}.    \]
 Combining the above inequalities, and taking $\eta$ as in the previous section,  we derive
  \begin{equation}
  \label{ggq-G150}
\begin{aligned}
\frac{c_1  |\nabla u|^2}{8 h^2} \sum G^{i\bi}
   + \frac{c_1}{8 h^2}  f (|\nabla u|^2 {\bf{1}}) 
     \leq \,& - \frac{1}{h} G^{i\bi}\bpartial_i \partial_i u  - \frac{1}{\eta} G^{i\bi}\bpartial_i \partial_i \eta 
    +  \frac{C}{\eta^2} G^{i\bi}\bpartial_i \eta \partial_i \eta \\
     &  +  \frac{c_1 \psi}{8 h^2} + C  \sum G^{i\bi} - \frac{R}{|\nabla u|^2}  \\
     \leq \,& \frac{1}{h} G^{i\bi} X_{i\bi} + \frac{c_1 \psi}{8 h^2} - \frac{R}{|\nabla u|^2}
              +  \frac{C}{\eta} \sum G^{i\bi}.
     \end{aligned}
\end{equation}  

Consequently, we obtain a bound $\eta |\nabla u|^2 \leq C/r^2$ if $\psi$ and $X$ are independent of $u$ and $\partial u$. 
In the general case, by \eqref{gblq-R155.6} and \eqref{gqy-G32},    
  \begin{equation}
\label{ggq-G160}
\begin{aligned}
\fRe\{\nabla_k \psi \nabla_{\bk}  u \} 
   = \, & \psi_u |\nabla u|^2         
           + \fRe\{\psi_k \nabla_{\bk} u + \psi_{\zeta_\alpha} \partial_{\alpha} |\nabla u|^2
           + \psi_{\zeta_\alpha} \Gamma_{\alpha k}^l \nabla_l u \nabla_{\bk} u\} \\  
  = \, &  |\nabla u|^2  (\psi_u - \fRe\{\psi_{\zeta_\alpha} \partial_{\alpha} \phi\})    
           + \fRe\{\psi_k \nabla_{\bk} u 
           + \psi_{\zeta_\alpha} \Gamma_{\alpha k}^l \nabla_l u \nabla_{\bk} u\} \\    
= \,&  |\nabla u|^2  A - \frac{|\nabla u|^2}{\eta}  
           \fRe\{\psi_{\zeta_\alpha} \partial_{\alpha}  \eta\}       
    \\
\geq  \,&  |\nabla u|^2  A - \frac{C |\nabla u|^2 |D_{\zeta} \psi|}{\sqrt{\eta}}   
 \end{aligned}
        \end{equation}    
and similarly, 
\begin{equation}
\label{ggq-G170}
 \begin{aligned}
G^{i\bi}  \fRe\{\nabla_{\bk} u \nabla_k X_{i\bi}\}  = \,& |\nabla u|^2 B 
   - \frac{|\nabla u|^2}{\eta} G^{i\bi} \fRe\{X_{i\bi, \zeta_{\alpha}} \partial_{\alpha} \eta\} \\
 \,& \leq  |\nabla u|^2 B 
   + \frac{C |\nabla u|^2 |D_{\zeta} X|}{\sqrt{\eta}} \sum G^{i\bi}   
 \end{aligned} 
\end{equation}
       where 
 \[ A = \psi_u - \frac{1}{h} \fRe\{\psi_{\zeta_\alpha} \partial_{\alpha} u\}
           + \frac{1}{|\nabla u|^2} \fRe\{\psi_k \nabla_{\bk} u
            + \psi_{\zeta_\alpha} \Gamma_{\alpha k}^l \nabla_l u \nabla_{\bk} u\}, \]                
\[ B =  G^{i\bi} X_{i\bi,u}  
          - \frac{1}{h} G^{i\bi} \fRe\{X_{i\bi, \zeta_{\alpha}} \partial_{\alpha} u\}
        + \frac{1}{|\nabla u|^2} G^{i\bi} \fRe\{(X_{i\bi k} 
         + X_{i\bi, \zeta_{\alpha}} \Gamma_{\alpha k}^l \nabla_l u) \nabla_{\bk} u\}. \]
By assumptions~\eqref{ggq-G10} and  \eqref{ggq-G20}, 
  \begin{equation}
  \label{ggq-G180}
    \begin{aligned}
   \frac{1}{h} G^{i\bi} X_{i\bi} + \frac{c_1 \psi}{8 h^2} - \frac{R}{|\nabla u|^2}
     \leq \, & C H \sum G^{i\bi} +  C E \\
               &  + \frac{C \varrho_0 |\nabla u|}{\sqrt{\eta}} \sum G^{i\bi} 
                   + \frac{C \varrho_0 f (|\nabla u|^2 {\bf{1}})}{ |\nabla u| \sqrt{\eta}} \\
     \end{aligned} 
\end{equation}
where
\[ E = |\nabla_z \psi| |\nabla u|^{-1} + (\psi_u)^-  + \psi^+ + |D_{\zeta} \psi| |\nabla u| 
            \leq  \varrho_0  f (|\nabla u|^2 {\bf{1}})  \]
by \eqref{ggq-G20} and \eqref{ggq-G25}, and 
\[ H = |\nabla_z X| |\nabla u|^{-1} + \tr X^+ + \tr (D_u X)^+  
            + |D_{\zeta} X| |\nabla u|  \leq \varrho_0 |\nabla u|^2 \]
by \eqref{ggq-G10} and \eqref{ggq-G15}. 
Plugging these back to \eqref{ggq-G180}, we obtain a bound $\eta |\nabla u|^2 \leq C/r^2$ at $z_0$ from which  \eqref{ggq-30} follows.
\end{proof}

\begin{remark}
\label{ggq-remark-G90}
Assume both $X$ and $\psi$ are independent of $u$. Taking $\eta = 1$ we obtain
\[ |\nabla u|^2  \leq C \Big(1 + \sup_M u - u\Big) \;\; \mbox{on M} \]
for any admissible solution $u \in C^3 (M)$ of equation~\eqref{ggq-I10Lambda}.
It follows that 
\[ \Big|\Big(1 + \sup_M u - u (x) \Big)^{\frac{1}{2}} - \Big(1 + \sup_M u - u (y) \Big)^{\frac{1}{2}}\Big| \leq C \mbox{dist} (x, y), \;\; \forall x, y \in M \]
 where $\mbox{dist} (x, y)$ denotes the distance 
 between $x$ and $y$ in $(M, \omega)$. In particular,
\begin{equation}
\label{ggq-G215}
 \sup_M u - \inf_M u \leq C d^2 
 \end{equation}
where $d$ is the diameter of $M$. 
 \end{remark}

\begin{remark}
\label{ggq-remark-G90}
Assume $u > 0$ on $B_r$ and take $h = u^2$ in the previous proof. We can derive
a bound $\eta |\nabla u|^2 \leq C/r^2$ at $z_0$ if $\psi$ and $X$ are independent of $u$ and $\partial u$.  This gives 
\[ \frac{|\nabla u|^2}{u^2} \leq C \;\; \mbox{in $B_{\frac{r}{2}}$}. \]
Consequently, we obtain the Harnack inequality
\[ \sup_{B_{\frac{r}{2}}} u \leq C \inf_{B_{\frac{r}{2}}} u. \]
 \end{remark}

 \begin{remark}
\label{ggq-remark-G100}
For a Hermitian manifold $(M^n, \omega)$ with boundary $\partial M$, 
taking $\eta = 1$ we derive the global gradient estimate
\[ \max_{\bM} |\partial u| \leq C \max_{\partial M} |\partial u| + C. \] 
\end{remark}

\bigskip

\section{Existence}
\label{ggq-E}
\setcounter{equation}{0}

In this section we outline proofs of Theorems~\ref{ggq-th-I30}, \ref{ggq-th-I40}.
With the aid of the {\em a priori} estimates in the specific forms derived in the previous sections we can extend these theorems 
to cover more general cases where $X$ and $\psi$ are allowed to depend on 
$u$ and $\nabla u$.

The following result extends Theorem~\ref{ggq-th-I40}.

\begin{theorem}
\label{ggq-th-I40'}
Let $(M^n, \omega)$ be a compact Hermitian manifold. Assume that
 \eqref{3I-30}, \eqref{3I-20w-Lambda}, 
\eqref{ggq-I50}, \eqref{ggq-I60Lambda}, \eqref{3I-40'}, \eqref{ggq-I45infty}, \eqref{ggq-G10}
and \eqref{ggq-G20} hold. In addition, assume that 
$X = X (z, \nabla u)$ and $\psi = \psi (z, \nabla u)$ are both independent of $u$,
and $X (z, 0) > 0$ on $M$. 
Then there exists a unique constant $b$ such that the equation 
\begin{equation}
\label{ggq-I10LC'}
  f (\Lambda (\sqrt{-1} \partial \bpartial u + X [u]))
      = e^b \psi [u]
\end{equation}
has a unique admissible solution $u \in C^{\infty} (M)$ up to a constant.
\end{theorem}

\begin{proof}
To begin with we assume $X [u] = X (z, u, \nabla u)$
and $\psi [u] = \psi (z, u, \nabla u)$ satisfying 
\begin{equation}
\label{ggq-E10}
 - D_u X, \; D_u \psi \geq 0. 
\end{equation}
First we choose $A > 0$ sufficiently large
such that $H \equiv f (\Lambda (A \omega)) > 0$. 
For $t \in [0, 1]$ consider the equation
\begin{equation}
\label{ggq-I10LCt}
  f (\Lambda (\sqrt{-1} \partial \bpartial u + X [u] + t A \omega))
      = e^{t u} \psi^t [u]
\end{equation}
where 
$\psi^t [u] =  
t H + (1-t) \psi [u]$.  
Note that for $t > 0$, by \eqref{ggq-E10}
\[ D_u (e^{t u} \psi^t [u]) \geq H t^2 e^{t u} > 0. \]
So the solution of equation~\eqref{ggq-I10LCt} is unique and its
linearized operator is nonsingular for $t > 0$. When $t = 1$ the solution is $u^1 = 0$. 

Suppose $u \in C^{\infty} (M)$ is the admissible solution of 
equation~\eqref{ggq-I10LCt} 
for fixed $t > 0$, and assume 
$\sup_M u \geq 0$. At a point where $u$ attains its maximum value,    
by \eqref{ggq-E10} we have
\[ \begin{aligned}
   e^{t u} \psi^t (z, 0, 0) \leq \,& e^{t u} \psi^t (z, u, 0) \\
   \leq \,& f (\Lambda (X (z, u, 0) + t A \omega) 
   \leq  f (\Lambda (X (z, 0, 0) + t A \omega). 
   \end{aligned}\]
Therefore,
\begin{equation}  
\label{ggq-E80}
t \sup_M u \leq
  \sup_{z \in M}  \log \frac{f (\Lambda (X (z, 0, 0) + A \omega)}{\psi^t (z, 0, 0)} \leq C. 
\end{equation}
Similarly, suppose $\inf_M u \leq 0$. Then 
\begin{equation}  
\label{ggq-E90}
t \inf_M u \geq \inf_{z \in M}  \log \frac{f (\Lambda (X (z, 0, 0))}{\psi^t (z, 0, 0)} \geq - C. 
\end{equation}
 Consequently, we may apply the continuity method to obtain a unique admissible solution $u^t \in C^{\infty} (M)$ of equation~\eqref{ggq-I10LCt} for all $t \in (0, 1]$ with the bound
 \begin{equation}  
\label{ggq-E100}
- C \leq t \inf_M u^t \leq t \sup_M u^t  \leq C, \;\; 0 < t \leq 1
\end{equation}
where $C$ is independent of $t$. 

By \eqref{ggq-E100} we can find a sequence $t_k \rightarrow 0$ such that
both of the following limits exist
\[ \lim_{k \rightarrow \infty} t_k \inf_M u^{t_k} = a, \;\;  
   \lim_{k \rightarrow \infty} t_k \sup_M u^{t_k} = b. \]

Suppose now that $X$ and $\psi$ are both independent of $u$. By \eqref{ggq-G215} we see that $a = b$ and therefore
\[  \lim_{k \rightarrow \infty} t_k u^{t_k} = b \;\; \mbox{on $M$}. \]
Moreover, by \eqref{ggq-G215} and the estimates we have established there exists a subsequence of
\[ \Big\{u^{t_k} - \sup_M u^{t_k}\Big\}  \]
converging in $C^{2, \alpha} (M)$
to a smooth admissible solution $u$ of equation~\eqref{ggq-I10LC'}. Clearly,
\[ \inf_M u \geq - C d^2,  \;\; \sup_M u = 0 \]
where $d$ is the diameter of $M$. 
\end{proof}

We next consider the Dirichlet problem
\begin{equation}
\label{ggq-I10LG}
 \left\{
\begin{aligned}
 f (\Lambda (\sqrt{-1} \partial \bpartial u + X [u])) = \,& \psi [u]  \; \mbox{in $\bM$}, \\
 u = \varphi \; \mbox{on $\partial M$}.
\end{aligned} \right.
\end{equation}
Here $(M^n, \omega)$ is assumed as in Theorem~\ref{ggq-th-I30}
to be a Hermitian manifold with smooth boundary 
$\partial M$ and compact closure $\bM = M \cup \partial M$, and 
 $\varphi \in C^0 (\partial M)$.

\begin{theorem}
\label{ggq-th-E30}
In addition to \eqref{3I-30}, \eqref{3I-40}, \eqref{3I-20w-Lambda}, 
\eqref{ggq-I50},  
\eqref{ggq-I60Lambda},  \eqref{ggq-G10}, \eqref{ggq-G20} and \eqref{ggq-E10},  assume that
there exists an admissible subsolution $\ul u \in C^0 (\bM)$ satisfying 
\begin{equation}
\label{ggq-I10Ls'}
 \left\{
\begin{aligned}
 f (\Lambda (\sqrt{-1} \partial \bpartial \ul u + X [\ul u])) \geq \,& \psi  [\ul u] \; \mbox{in $\bM$}, \\
\ul u = \varphi \; \mbox{on $\partial M$}
\end{aligned} \right.
\end{equation}
in the viscosity sense.
The Dirichlet problem~\eqref{ggq-I10LG} 
admits a unique admissible solution $u \in C^{\infty} (M) \cap C^0 (\bM)$.
Moreover, $u \in C^{\infty} (\bM)$ if $\varphi \in C^{\infty} (\partial M)$.
\end{theorem}

\begin{proof}
We assume $\ul u \in C^{\infty} (\bM)$ and $\varphi \in C^{\infty} (\partial M)$; the general case follows from approximation.  
Let $u \in C^{4} (M) \cap C^0 (\bM)$ be an admissible solution of 
problem~\eqref{ggq-I10LG}. 
From the comparison principle and \eqref{ggq-E10} we see that $u \geq \ul u$ on $M$
and therefore
\begin{equation}  
\label{ggq-E30}
 \frac{\partial u}{\partial \nu} \leq \frac{\partial \ul u}{\partial \nu} 
     \;\; \mbox{on $\partial M$} 
  \end{equation}
where $\nu$ denotes the exterior unit normal to $\partial M$. 
By assumptions \eqref{ggq-G10} and \eqref{ggq-E10},
\[ \tr X [u] \leq \tr X (z, \inf_M \ul u, \partial u, \bpartial u) 
     \leq \varrho_1 (z, \inf_M \ul u) + \varrho_0 (z, \inf_M \ul u, |\partial u|) |\partial u|^2. \]
Since $\Delta u + \tr X [u] \geq 0$, we have
\begin{equation}  
\label{ggq-E120}
\Delta u + \tr X (z, \inf_M \ul u, \partial u, \bpartial u) \geq 0. 
 \end{equation}
It follows from the maximum principle that $u \leq h$ on $\bM$ where $h \in C^{\infty} (\bM)$ is the unique solution of   
\begin{equation}  
\label{ggq-E130}
\Delta h + \tr X (z, \inf_M \ul u, \partial h, \bpartial h) = 0 \;\;
    \mbox{in $M$},  \;\; h = \varphi \;\; \mbox{on $\partial M$}.
 \end{equation}
 By the assumption $\varrho_0 = \varrho_0 (z, u, t) \rightarrow 0^+$ as $t \rightarrow \infty$ in \eqref{ggq-G10} we see that problem~\eqref{ggq-E130} is solvable. 
Combining 
with the gradient estimate in Section~\ref{ggq-G}
(see Remark~\ref{ggq-remark-G100}), we have established the global $C^1$ bound 
\begin{equation}  
\label{ggq-E50}
 \sup_{\bM} |u| + \sup_{\bM} |\nabla u| \leq C. 
  \end{equation}
By Theorem~\ref{ggq-th-I20} and Evans-Krylov Theorem we obtain the $C^{2, \alpha}$ estimate
\begin{equation}  
\label{ggq-E60}
 |u|_{C^{2, \alpha} (\bM)} \leq C. 
  \end{equation}
The existence and smoothness of a unique admissible solution now follows from the standard continuity method and the classical Schauder theory for linear uniformly elliptic equations.
\end{proof}

\begin{remark}
\label{ggq-remark-E10}
The assumption $\varrho_0 = \varrho_0 (z, u, t) \rightarrow 0^+$ as $t \rightarrow \infty$ in \eqref{ggq-G10} is critical to the existence of solution of 
problem~\eqref{ggq-E130}.

Consider the problem in a domain $\Omega \subset \bfR^n$
\begin{equation}  
\label{ggq-E135}
\Delta h + |\nabla h|^2 + b = 0 \;\;
    \mbox{in $\Omega$},  \;\; h = 0 \;\; \mbox{on $\partial \Omega$}.
 \end{equation}
Clearly $h$ satisfies $\Delta e^h + b e^h = 0$. Therefore, problem~\eqref{ggq-E135} has no solution if $b$ is the first eigenvalue of $\Omega$. Indeed, let $\phi_1$ be
the first eigenfunction 
\begin{equation}  
\label{ggq-E140}
\Delta \phi_1 + b \phi_1 = 0 \;\;
    \mbox{in $\Omega$},  \;\; \phi_1 = 0 \;\; \mbox{on $\partial \Omega$}, 
 \end{equation}
normalized so that
\[    \sup_{\Omega} \phi_1 = 1. \]
Note that $e^h > 1 \geq \phi_1$ in $\Omega$ since $\Delta e^h < 0$ in $\Omega$. We can find $t > 1$ such that 
\[ 0 < \inf_{\Omega} (e^h - t \phi_1) < 1. \]
Consequently, in the interior point $x_0$ where $e^h - t \phi_1$ attains the minimal value, 
\[ \Delta (e^h - t \phi_1) +b (e^h - t \phi_1) \geq b (e^h - t \phi_1) > 0 \]
which is a contradiction. 

When $u$ is a radial function, in spherical coordinates equation~\eqref{ggq-E135}
reduces to the simple Riccati equation for $y = u'$
\[ y' + \frac{n-1}{r} y + y^2 + b = 0. \]
In the special case $n = 1$ and $b = 1$, the solution is 
$u (x) = - \log \cos x$
which is only defined for $- \frac{\pi}{2} < x < \frac{\pi}{2}$. 
\end{remark}

\bigskip

\bigskip
\small
\bibliographystyle{plain}

\end{document}